\documentclass[a4paper, 11pt]{amsart}
\input xy
\xyoption{all}
\usepackage{amsmath,amsthm,amssymb,color}
\usepackage{graphicx}

\theoremstyle{plain}
\newtheorem{theorem}{Theorem}
\newtheorem{prop}{Proposition}
\theoremstyle{definition}
\newtheorem{definition}{Definition}
\newtheorem{example}{Example}

\newtheorem{remark}[definition]{Remark}

\newtheorem{assumption}{Assumption}

\newcommand{\cD}{\mathcal{D}}

\newcommand{\cF}{\mathcal{F}}

\newcommand{\cU}{\mathcal{U}}

\newcommand{\cf}{{\it cf.}\/ }

\newcommand{\X}{\mathcal{X}}

\newcommand{\cX}{\mathcal{X}}

\newcommand{\N}{\ensuremath{\mathbb N}}

\newcommand{\R}{\ensuremath{\mathbb R}}

\begin{document}

\title[Laplacians on generalized smooth distributions]{Laplacians on generalized smooth distributions as $C^*$-algebra multipliers}

\author{Iakovos Androulidakis}
 
\address{Department of Mathematics\\ National and Kapodistrian University of Athens\\ Panepistimiopolis\\ GR-15784 Athens\\ Greece}
 
\email{iandroul@math.uoa.gr}
\thanks{}

\author{Yuri A. Kordyukov}
 
\address{Institute of Mathematics\\ Ufa Federal Research Centre\\ Russian Academy of Science\\ Chernyshevsky str.\\ 450008 Ufa\\ Russia}
 
\email{yurikor@matem.anrb.ru}
\thanks{}



\date{}

\begin{abstract}
In this paper, we discuss spectral properties of Laplacians associated with an arbitrary smooth distribution on a compact manifold. First, we give a survey of results on generalized smooth distributions on manifolds, Riemannian structures and associated Laplacians. Then, under the assumption that the singular foliation generated by the distribution is regular, 
we prove that the Laplacian associated with the distribution defines an unbounded multiplier on the foliation $C^*$-algebra. To this end, we give the construction of a parametrix.
\end{abstract}

\maketitle

\tableofcontents

\section{Introduction}\label{s:intro}

The purpose of this paper is to discuss spectral properties of Laplacians associated with an arbitrary smooth distribution on a compact manifold. This includes smooth distributions of non constant rank. In fact, these singular distributions are our main focus,  as they arise in sub-Riemannian Geometry.

Recall that  the case of a constant rank distribution $D$ was studied by one of us in \cite{YK1}, \cite{YK2}. This was achieved by considering the Laplacians involved as operators in the longitudinal pseudodifferential calculus of the smallest singular foliation $\cF$ which includes $D$. This calculus was developed by one of us and G. Skandalis in \cite{AS2}, making use of the notion of bisubmersion \cite{AS1}. Also recall that in \cite{AK1} we gave a geometric construction of Laplacians for arbitrary distributions and showed their self-adjointness and hypoellipticity properties.

In this paper, we start with a survey of the results in \cite{AK1}. We recall that modules of vector fields are the appropriate context for the study of generalized smooth distributions on manifolds; then we describe the construction of Riemannian structures and their associated Laplacians. The notion of local presentation of a module of vector fields makes all these developments possible.

Then, assuming that the singular foliation generated by the distribution is regular, we prove that the Laplacian associated with the distribution defines an unbounded multiplier on the foliation $C^*$-algebra. As shown in \cite{Connes79}, the spectral properties of our Laplacians are a consequence of this result. Note that our assumption is justified in view of the distributions arising in sub-Riemannian geometry, which have the bracket generating property.

The proof requires the explicit construction of a parametrix. We show that the construction given by Rothschild and Stein \cite{Rothschild-Stein} can be applied in our case. This result has independent interest.

\section{Distributions as modules of vector fields}\label{s:multi}

We start with the definition for distributions in terms of vector fields, which focuses more on the dynamics involved. It is inspired by the definition of a singular foliation in \cite{AS1}. 

Throughout the paper, $M$ is a smooth manifold with dimension $n$. For a smooth manifold $N$, we denote by $\cX(N)$ (resp. $\cX_c(N)$) the $C^{\infty}(N)$-module of smooth (resp. compactly supported smooth) vector fields on $N$.

Let $\cD$ be a $C^{\infty}(M)$-submodule of $\cX_c(M)$:

\begin{enumerate}
\item  Given an open subset $U$ of $M$, we  put $\iota_U : U \hookrightarrow M$ the inclusion map, and, for a vector field $X \in \cX(M)$, write $X\left|_U\right. = X \circ \iota_U$. The \emph{restriction} of $\cD$ to $U$ is the $C^{\infty}(U)$-submodule of $\cX_c(U)$ generated by $f\cdot X\left|_U\right.$, where $f \in C^{\infty}_c(U)$ and $X \in \cD$. We denote this restriction $\cD\left|_U\right.$.

\item We say that the module $\cD$ is \emph{locally finitely generated} if, for every $x \in M$, there exist an open neighborhood $U$ of $x$ and a finite number of vector fields $X_1,\ldots,X_k$ in $\cX(M)$ such that $\cD\left|_U\right. = C^{\infty}_c(U) \cdot X_1\left|_U\right. + \ldots + C^{\infty}_c(U) \cdot X_k\left|_U\right.$. We say that the vector fields $X_1,\ldots,X_k$ \emph{generate} $\cD\left|_U\right.$ or they are \emph{local generators} of $\cD$.
\end{enumerate}

\begin{definition}\label{dfn:distr}
A (generalized) smooth distribution on $M$ is a locally finitely generated $C^{\infty}(M)$-submodule $\cD$ of the $C^{\infty}(M)$-module $\cX_c(M)$. We denote a distribution as a pair $(M,\cD)$.

A \emph{singular foliation} is a smooth distribution $\cF$ on $M$ which is involutive, namely $[\cF,\cF]\subseteq \cF$.
\end{definition}

\begin{example}
An arbitrary action of a finite-dimensional Lie group on $M$ defines a singular foliation.
\end{example}

\begin{example}\label{ex:nonfoliation}
Let $f \in C^{\infty}(\R^2)$ be defined by $f(x,y) = e^{-\frac{1}{x}}$ if $x >0$ and $f(x,y)=0$ if $x\leq 0$. Consider the smooth distribution $\cF$ on $\R^2$ generated by the vector fields $X= \partial_x$ and $Y_n = x^{-n}f(x,y)\partial_y$ for all $n \in \N$. then $\cF$ is involutive, but not (locally) finitely generated.
\end{example}

Let $(M,\cD)$ be a smooth distribution. There is a naive way to associate with $\cD$ a distribution in the usual sense. For any $x\in M$, let $ev_x : \cD \to T_x M$ be the linear map given by the evaluation at $x$. Put $D_x$ the image of this map. It is a vector subspace of $T_x M$. The field of vector spaces $\cup_{x \in M}D_x$ is a distribution on $M$ in the usual sense. If the dimension of $D_x$ is locally constant, then $D$ is a smooth distribution in the usual sense, that is, a vector subbundle of $TM$, and $\cD$ is a projective $C^\infty(M)$-module. The converse is also true \cite{AZ5}.   

There is a more natural way to associate with $\cD$ a field of vector spaces.
For any $x \in M$, consider the $C^{\infty}(M)$-submodule $I_x\cD$ of $\cX_c(M)$, where $I_x = \{f \in C^{\infty}(M) : f(x)=0\}$. Since $\cD$ is locally finitely generated, the quotient $\cD_x = \cD/I_x\cD$ is a finite dimensional vector space. We call it the \textit{fiber} of $(M,\cD)$ at $x$. For any $X\in \cD$, we will denote by $[X]_x$ the corresponding class in $\cD_x$. 

The fibers $\cD_{x}$ provide a way to find a minimal set of local generators of $\cD$. More precisely (\cf \cite[Prop. 1.5]{AS1}), if $X_1,\ldots,X_\ell \in \cD$ are such that their restrictions to some open subset $U$ generate $\cD\left|_U\right.$, then, for any $x\in U$, we have $\dim \cD_x\leq \ell$. On the other hand, if $X_1,\ldots,X_k \in \cD$ are such that their images in $\cD_x$ give a basis of $\cD_x$, then there exists a neighborhood $U$ of $x$ such that the restrictions of $X_1, \ldots, X_k$ to $U$ generate $\cD\left|_U\right.$.

\begin{example}
Let us consider the distribution $\cD$ on $M=\R^2$ generated by the vector fields $X_1 = \partial_x$ and $X_2=x\partial_y$  (Grushin plane). It is easy to see that, for $(x,y)\in \R^2$ with $x \neq 0$, we have $\cD_{(x,y)}=D_{(x,y)}=\R^2$ and, for a point on the $y$-axis, we have $\cD_{(0,y)} = \R^2$ and $D_{(0,y)} = \R$ for $y \in \R$. 
\end{example}

As shown in \cite[Prop. 1.5]{AS1}, the dimension map $\dim_{\cD} : M \to \N, x \mapsto \dim \cD_x$ is upper semicontinuous, and the dimension map $\dim_D : M \to \N, x \mapsto \dim D_x$ is lower semicontinuous. The set of continuity of $\dim_D$ is 
\[
\mathcal{C} = \{x \in M : ev_x : \cD_x \to D_x \text{ is bijective } \}.
\] 
It is an open and dense subset of $M$. The restriction $\cD\left|_{\mathcal{C}}\right.$ is a projective $C^{\infty}(\mathcal{C})$-submodule of $\cX(\mathcal{C})$, whence it is the module of sections of a vector subbundle $D$ of $T\mathcal{C}$.

We recall from \cite{AK1} a notion of local presentation for a smooth distribution $\cD$, which provides a geometric reformulation of the algebraic assumption on $\cD$ to be locally finitely generated. 

First, recall that an anchored vector bundle over $M$ is a vector bundle $E \to M$ endowed with a morphism of vector bundles $\rho : E \to TM$ over the identity diffeomorphism of $M$. An anchored vector bundle $\rho : E \to TM$ over $M$ induces a morphism of $C^{\infty}(M)$-modules $\Gamma_c E \to \cX_c(M)$, which we also denote $\rho$ by abuse of notation. Then the module $\cD_E = \rho(\Gamma_c E)$ is locally finitely generated: Indeed, if $\sigma_1,\ldots,\sigma_k$ is a frame of $E$ over an open $U \subset M$, the module $\cD_E\left|_U\right.$ is generated by the restrictions to $U$ of the vector fields $X_i = \rho(\sigma_i)$, $1 \leq i \leq k$. Whence $(M,\cD_E)$ is a smooth distribution.

Conversely, let $(M,\cD)$ be a distribution. Take any $X_1,\ldots,X_k\in \cD$ such that $X_1\left|_U\right.,\ldots,X_k\left|_U\right.$ generate $\cD\left|_U\right.$ in some open subset $U$. Consider an anchored vector bundle $\rho_U : E_U \to TU$ over $U$, where $E_U$ is the trivial bundle $U \times \R^k$ and $\rho_U : E_U \to TU$ is the map $\rho_U(y,\lambda_1,\ldots,\lambda_k)=\lambda_1 X_1(y) + \ldots + \lambda_k X_k(y)$. At the level of sections, we obtain the map $\rho_U : \Gamma_c E_U=C^{\infty}_c(U)^k \to \cX_c(U)$ given by 
\[
\rho_U(f_1,\ldots,f_k)=f_1 \cdot X_1\left|_U\right. + \ldots + f_k \cdot X_k\left|_U\right.
\]
such that $\rho_U(\Gamma_c E_U)=\cD\left|_U\right.$. This motivates the following definition.

\begin{definition}\label{dfn:presentation}
Let $(M,\cD)$ be a distribution and $U$ an open subset of $M$. 
\begin{enumerate}
\item A \emph{local presentation} of $(M,\cD)$ over $U$ is an anchored vector bundle $\rho_U : E_U \to TU$ over $U$, such that 
\[ 
\rho_U(\Gamma_c E_U) = \cD\left|_U\right..
\]
Once the distribution $(M,\cD)$ is fixed, a local presentation as such is denoted $(E_U,\rho_U)$. 
\item Let $W$ be an open subset of $U$. A \emph{morphism of local presentations} from $(E_U,\rho_U)$ to $(E_W,\rho_W)$ is a vector bundles morphism $\psi : E_U \left|_W\right. \to E_W$ (over the identity) such that $\rho_W \circ \psi = \rho_U$. A morphism of local presentations from $(E_W,\rho_W)$ to $(E_U,\rho_U)$ is a morphism of vector bundles $\phi : E_W \to E_U$ over the inclusion $\iota : W \hookrightarrow U$ such that $\rho_U \circ \phi = \rho_W$.
\end{enumerate}
\end{definition}

Given a local presentation $(E_U,\rho_U)$, fix a point $x$ in $U$. Recall from the Serre-Swan theorem that the fiber $(E_U)_x$ is the quotient of the $C^{\infty}(U)$-module $\Gamma_c E_U$ by the $C^{\infty}(U)$-submodule $I_x \Gamma_c E_U$ (\cf \cite{AZ5}). Since $\rho(I_x\Gamma_c E_U)\subseteq I_x\cD\left|_U\right.$, we obtain a linear epimorphism 
\[
\widehat{\rho}_{U,x} : (E_U)_x \to \cD\left|_U\right./I_x\cD\left|_U\right.=\cD_x.
\] 
Whence the dimension of the fiber $\cD_x$ at any $x \in U$ is bounded above by the rank of $E_U$. Composing $\widehat{\rho}_{U,x}$ with the evaluation map we recover the restriction of $\rho_U$ to the fiber $(E_U)_x$. This is a linear epimorphism $\rho_{U,x} : (E_U)_x \to D_x$. Whence the following diagram commutes:
\[
\xymatrix{
(E_U)_x \ar@{->>}[r]^{\widehat{\rho}_{U,x}} \ar@{->>}[rd]_{\rho_{U,x}} & \cD_x \ar@{->>}[d]^{ev_x} \\
& D_x
}
\]

Recall that, for any local presentation $(E_U,\rho_U)$ of $(M,\cD)$ over an open subset $U$ and for any $x\in U$, we have $\operatorname{rank} E_U\geq \dim \cD_x$. A local presentation $(E_U,\rho_U)$ is said to be \emph{minimal} at $x$ if $\operatorname{rank} E_U= \dim \cD_x$ (or, equivalently, if the linear epimorphism $\widehat{\rho}_{U,x} : (E_U)_x \to \cD_x$ is an isomorphism). As mentioned above, a minimal local presentation at $x \in M$ can be constructed, starting from a basis of $\cD_x$. Minimal local presentations play an essential role in several proofs.   

One can define the following compatibility relation between different local presentations. 

\begin{definition}\label{dfn:equiv}
Let $(M,\cD)$ be a distribution and $U, V$ open subsets of $M$ such that $U \cap V \neq \emptyset$. Two local presentations $(E_U,\rho_U)$ and $(E_V,\rho_V)$ are called \textit{equivalent at a point $x \in U \cap V$}, if there exist a local presentation $(E_W,\rho_W)$ over an open neighbourhood  $W$ of $x$ such that $W \subset U \cap V$ and morphisms of local presentations $\phi_{W,U} : (E_W,\rho_W) \to (E_U,\rho_U)$ and $\phi_{W,V} : (E_W,\rho_W) \to (E_V,\rho_V)$ such that $\rho_U\left|_W\right. \circ \phi_{W,U}=\rho_W = \rho_V\left|_W\right. \circ \phi_{W,V}$. 
\end{definition}

One can show \cite[\S 1.3]{AK1} that the compatibility relation introduced in Definition \ref{dfn:equiv} is an equivalence relation. The following proposition \cite[Prop. 1.15]{AK1} suggests that the set of all local representations equipped with this relation can be considered as a (maximal) atlas for the distribution.

\begin{prop}\label{prop:equivminlocpr}
Suppose that $U, V$ are open subsets of $M$ such that $U \cap V \neq \emptyset$. Then any local presentations $(E_U,\rho_U)$ and $(E_V,\rho_V)$ are equivalent at every $x \in U\cap V$.
\end{prop}

\section{The horizontal differential}\label{sec:hordif}
Given a smooth distribution $(M,\cD)$, denote $\cD^*$ the disjoint union of vector spaces $\bigsqcup_{x\in M}\cD^*_x$. Recall that in \cite[Prop. 2.10]{AS2}, it was shown that $\cD^{\ast}$ is a locally compact space. Its topology is the smallest topology which makes the following maps continuous:
\begin{itemize}
\item $p : \cD^{\ast} \to M$ is the projection $p(x,\xi)=x$.
\item For every $X \in \cD$ the map $q_X : \cD^{\ast} \to \R$ with $q_X(x,\xi)=\langle \xi, [X]_x \rangle$.
\end{itemize}

First, with the help of local presentations, we make sense of the smooth sections of this family of vector spaces. 

 
\begin{definition}\label{dfn:smoothdual}
Let $\omega^{\ast}$ be a map $M \ni x \mapsto \omega^{\ast}(x) \in \cD^{\ast}_x$. We say that $\omega^{\ast}$ is a \emph{smooth} section of $\cD^{\ast}$ iff for every $x \in M$ there is a local presentation $(E_U,\rho_U)$ defined in a neighborhood of $x$ such that the section $\omega^{\ast}_U$ of the bundle $E_U^{\ast}$ (a \emph{local realization} of $\omega^{\ast}$) defined by the commutative diagram 
\[
\xymatrix{
 & \omega^{\ast}_U(y) \in E^{\ast}_{U,y} \\ 
y\in U  \ar[r]_{\omega^{\ast}} \ar[ru]|-{\omega^{\ast}_U} & \omega^{\ast}(y) \in \cD^{\ast}_y \ar[u]_{\widehat{\rho}_{U,y}^{\ast}} 
}
\]
or, equivalently, by $\omega^{\ast}_U(y)=\widehat{\rho}^{\ast}_{U,y} \circ \omega^{\ast}(y)$ for all $y \in U$ is smooth on $U$:

Note that, since $\widehat{\rho}_{U,x}$ is surjective, its dual map $\widehat{\rho}^{\ast}_{U,x}$ is injective.
\end{definition}

We denote the set of smooth sections of $\cD^{\ast}$ by $C^{\infty}(M,\cD^{\ast})$ and its subset consisting of sections with compact support by $C^{\infty}_c(M,\cD^{\ast})$. Regarding the definition of the $C^{\infty}(M)$-module structure for $C^{\infty}_c(M,\cD^{\ast})$, it is $(f\cdot\omega^{\ast})(y) = f(y)\cdot \omega^{\ast}(y)$. Note that if $\omega^{\ast}_U$ is a local realization of $\omega^{\ast}$ then $f \left|_U\right. \cdot \omega^{\ast}_U$ is a local realization of $f\cdot \omega^{\ast}$.
 
Smooth sections as such can also be characterized \cite[Prop. 3.4]{AK1} in a coordinate-free fashion:

\begin{prop}\label{not:smoothdual}
Let $\omega^{\ast}$ be a map $M \ni x \mapsto \omega^{\ast}(x) \in \cD^{\ast}_x$. If $\omega^*\in C^\infty(M,\cD^*)$, then the function $M \ni x \mapsto \langle \omega^{\ast}(x), [X]_x \rangle$ is smooth on $M$ for any $X\in\cD$. Conversely, if the function $M \ni x \mapsto \langle \omega^{\ast}(x), [X]_x \rangle$ is smooth on $M$ for any $X\in\cD$ and $(E_V,\rho_V)$ is an arbitrary local presentation of $\cD$, then the local realization $\omega^{\ast}_V$ of $\omega^{\ast}$ is smooth on $V$. 
\end{prop}

As a consequence, we get a well-defined map 
$ev^{\ast} : \Omega^1_c(M) \to C^{\infty}_c(M,\cD^{\ast})$
induced by the evaluation maps $ev_x : \cD \to T_x M$, $x\in M$. For any $\omega\in \Omega^1_c(M)$ and $x\in M$, its image $ev^{\ast}\omega(x)\in \cD^{\ast}_x$ is given by
\[
\langle ev^{\ast}\omega(x), [X]_x\rangle = \langle \omega(x), X(x)\rangle,\quad  X\in \cD.
\]
By Proposition \ref{not:smoothdual}, it is clear that the function $M \ni x \mapsto ev^{\ast}\omega(x) \in \cD^{\ast}_x$ is smooth. 

We are now ready to give the definition of the horizontal differential of a distribution.

\begin{definition}\label{dfn:hordif}
Let $(M,\cD)$ be a smooth distribution.  
\begin{enumerate}
\item The  \textit{horizontal differential} is the operator  
$d_{\cD} : C^{\infty}_c(M) \to C^{\infty}_c(M,\cD^{\ast})$ defined as $d_{\cD} = ev^{\ast}\circ d$, where $d : C^{\infty}_c(M) \to \Omega^1_c(M)$ is the de Rham differential.
\item Given a local presentation $(E_U,\rho_U)$, put $d_{E^{\ast}_U} : C^{\infty}_c(U) \to C^{\infty}_c(U,E^{\ast}_U)$ the operator defined as the composition of the de Rham differential $d : C^{\infty}_c(U) \to \Omega^1_c(U)$ with the map $\rho_U^{\ast} : \Omega^1_c(U) \to C^{\infty}_c(U,E_U^{\ast})$. We call $d_{E^{\ast}_U}$ a \emph{local presentation} of the horizontal differential $d_{\cD}$.
\end{enumerate}
\end{definition}

Note that the terminology ``local presentation'' for the operator $d_{E^{\ast}_U}$ is justified by the following commutative diagram:
\[
\xymatrix{
C^{\infty}_c(U) \ar[r]^{d} & \Omega^1_c(U)  \ar[rd]_{ev^{\ast}} \ar[rr]^{\rho^{\ast}_U} & & C^{\infty}_c(U,E^{\ast}_U) \\
 & & C^{\infty}_c(U,\cD^{\ast}) \ar[ru]_{\widehat{\rho}^{\ast}_U} &
}
\]
Thus, we have
\[
d_{E^{\ast}_U}=\widehat{\rho}^{\ast}_U\circ d_{\cD}. 
\]

\section{Riemannian metric on a distribution}\label{sec:dfnRiem}

Here we will give the definition of Riemannian structure on a distribution, extending the classical definition of Riemannian structure on a vector bundle. So a Riemannian metric on a distribution $(M,\cD)$ needs to be defined on a family of pointwise linearizations of $\cD$, and must be smooth in some sense. The fibers $\cD_x = \cD / I_x \cD$ play the role of these linearizations, and the local presentations of $\cD$ can be used to make sense of this smoothness. But first we need the following, quite classical, facts:

\begin{enumerate}
\item Suppose that $(E, \langle \cdot, \cdot  \rangle_{E})$ and $(F, \langle \cdot, \cdot  \rangle_{F})$ are two (finite dimensional) Euclidean vector spaces and $A: E\to F$ is a linear epimorphism. Then we have the induced linear map $\bar A: E/\ker A \to F$, which is an isomorphism. 

The inner product $\langle \cdot, \cdot  \rangle_{E}$ induces an inner product $\langle \cdot, \cdot  \rangle_{E/\ker A}$ on $E/\ker A$, using the isomorphism $E/\ker A\cong (\ker A)^\bot$. 

We say that $A$ is a {\it Riemannian submersion}, if $\bar A$ preserves inner products:
\[
\langle \bar A u, \bar Av  \rangle_{F}= \langle u, v  \rangle_{E/\ker A}, \quad u,v \in E/\ker A. 
\] 

\item If $A: E\to F$ is a linear epimorphism and $\langle \cdot, \cdot  \rangle_{E}$ is an inner product on $E$, then there exists a unique inner product $\langle \cdot, \cdot  \rangle_F$ on $F$ such that $A : (E, \langle \cdot, \cdot  \rangle_{E})\to (F, \langle \cdot, \cdot  \rangle_{F})$ is a Riemannian submersion. This follows immediately from the fact that the induced map $\bar A: E/\ker A \to F$ is an isomorphism. The corresponding norm is given by 
\[
\|u\|_{F}= \|\bar A^{-1} u\|_{E/\ker A}=\inf \{\|w\|_E : w\in E, Aw=u\}, \quad u\in F. 
\]

\item If $(E, \langle \cdot, \cdot  \rangle_{E})$ and $(F, \langle \cdot, \cdot  \rangle_{F})$ are two Euclidean vector spaces and $A: E\to F$ is a linear epimorphism, then the adjoint $A^*: F \to E$ is a linear monomorphism. One can check that $A$ is a Riemannian submersion if and only if $A^*$ is an isometry, that is, preserves inner products:
\[
\langle A^* u, A^*v  \rangle_E= \langle u, v  \rangle_{F}, \quad \text{ for all } u,v \in F. 
\] 
\end{enumerate}

\begin{definition}\label{dfn:metric}
A \emph{Riemannian metric} on $(M,\cD)$ is a family 
\[
\langle\ ,\ \rangle_{\cD}=\{\langle\cdot, \cdot\rangle_x,  x\in M\}
\] 
of Euclidean inner products $\langle\cdot, \cdot\rangle_x$ on $\cD_x$, which is smooth in the following sense. For every $x \in M$, there exist an open neighborhood $U$ of $x$, a local presentation $\rho_U : E_U \to TM$ of $(M,\cD)$, and a smooth family of inner products $\{\langle\cdot, \cdot\rangle_{(E_U)_y}, y\in U\}$ in the fibers of $E_U$, such that, for any $y\in U$, the linear epimorphism $\hat{\rho}_{U,y} : (E_U)_y \to \cD_y$ is a Riemannian submersion. 

A local presentation $(E_U,\rho_U)$ as above is called a local presentation of the Riemannian metric $\langle\ ,\ \rangle_{\cD}$ over $U$. 
\end{definition}

\begin{theorem}[\cite{AK1}]
Let $(M,\cD)$ be an arbitrary smooth distribution. There exists a Riemannian structure for $(M,\cD)$.
\end{theorem}

\section{The horizontal Laplacian of a distribution}\label{sec:horLapl}
A naive approach to introducing an adjoint for the horizontal differential $d_{\cD}= ev^{\ast}\circ d$ of a distribution $(M,\cD)$ would be to use a Riemannian metric on $M$ in order to make sense of the adjoint of the usual de Rham differential $d^{\ast}$. But such a metric would have to be somehow compatible with the Riemannian metric of the distribution $(M,\cD)$, and this reduces considerably the range of applicability of our constructions.

Instead, we will show in this section that an adjoint can be constructed only with the data of the Riemannian metric on the distribution and the smooth density of $M$, for which no compatibility is required. This is possible thanks to the local presentations of our Riemannian metric.

Let us fix a Riemannian metric $\langle\ ,\ \rangle_{\cD}=\{\langle\cdot, \cdot\rangle_x, x\in M\}$ on the distribution $(M,\cD)$, as in Definition \ref{dfn:metric}, and a positive smooth density $\mu$ on $M$.
Then one can define a family $\langle\ ,\ \rangle_{\cD^*}=\{\langle\cdot, \cdot\rangle_{x}, x\in M\}$ of inner products on $\cD^*_x$ and the pointwise inner product of two elements $\omega,\omega^\prime \in C^\infty(M,\cD^*)$ as a function $\langle \omega, \omega^\prime \rangle_{\cD^*}$ on $M$ given by
\[
\langle \omega, \omega^\prime \rangle_{\cD^*}(x)=\langle \omega(x), \omega^\prime(x) \rangle_x,\quad x\in M.
\] 
We can also define an inner product on $C^\infty_c(M,\cD^*)$ by
\[
(\omega,\omega^\prime)_{L^2(M,\cD^*,\mu)}=\int_M \langle \omega, \omega^\prime \rangle_{\cD^*}(x) d\mu(x),\quad \omega,\omega^\prime\in C^\infty(M,\cD^*).
\]
An important observation is that, for any $\omega,\omega^\prime\in C^\infty(M,\cD^*)$, one can show $\langle \omega, \omega^\prime \rangle_{\cD^*}\in C^\infty(M)$, so the integral is well-defined.

Since $\cD^*$ is not a vector bundle, the existence of the adjoint 
\[
d_{\cD}^{\ast} : C^{\infty}_c(M,\cD^{\ast}) \to C^{\infty}_c(M)
\] 
of the operator $d_{\cD} : C^{\infty}_c(M) \to C^{\infty}_c(M,\cD^{\ast})$ is not immediate. Such an adjoint arises from the adjoints of local presentations of $d_\cD$. More precisely, let $U$ be an open subset of $M$ and $(E_U,\rho_U)$ be a local presentation of the Riemannian metric on $(M,\cD)$. First we can define an inner product on $C^{\infty}_c(U,E_U^{\ast})$ by 
\[
(\omega^{\ast}_{1},\omega^{\ast}_{2})_{L^2(U,E^{\ast}_U,\mu)} = \int_U \langle \omega^{\ast}_{1}(y),\omega^{\ast}_{2}(y) \rangle_{E^\ast_{U,x}}d\mu(y).
\] 
Since $d_{E^{\ast}_U}$ is a first order differential operator, acting in sections of vector bundles, there exists its adjoint $d_{E^{\ast}_U}^{\ast} : C^{\infty}_c(U,E^{\ast}_U) \to C^{\infty}_c(U)$, which is a first order differential operator, satisfying  
\[
(d_{E^{\ast}_U}^{\ast}\omega^{\ast},\alpha )_{L^2(U,\mu)} = (\omega^{\ast},d_{E^{\ast}_U}\alpha)_{L^2(U,E^{\ast}_U,\mu)}
\]
for all $\omega^{\ast} \in C^{\infty}_c(U,E^{\ast}_U)$ and $\alpha \in C^{\infty}_c(U)$.

For $\omega^{\ast} \in C^{\infty}_c(U,\cD^{\ast})$, we define $d^{\ast}_{\cD,U} \omega^{\ast}\in C^\infty_c(U)$ by
\[
d^{\ast}_{\cD,U} \omega^{\ast} (y) = d_{E^{\ast}_U}^{\ast} \omega^{\ast}_U(y), \quad y \in U,
\]
where $\omega^{\ast}_U \in C^{\infty}_c(U,E^{\ast}_U)$ is the local realization of $\omega^{\ast}$. 

One can show that the resulting operator $d^{\ast}_{\cD,U} : C^{\infty}_c(U,\cD^{\ast}) \to C^{\infty}_c(U)$ is well-defined and is the adjoint of $d_{\cD}\left|_{U}\right.$. Moreover, these locally defined adjoints $d^{\ast}_{\cD,U}$ agree with each other on intersections, giving rise to a global operator $d^{\ast}_{\cD} : C^{\infty}_c(M,\cD^{\ast}) \to C^{\infty}_c(M)$, which is the adjoint of $d_{\cD}$.  

Now we are able to define the horizontal Laplacian of a distribution.

\begin{definition}\label{dfn:horlapl}
Let $(M,\cD)$ be a smooth distribution. Choose a Riemannian metric on $\cD$ and a positive smooth density $\mu$ on $M$. The operator $\Delta_{\cD} = d^{\ast}_{\cD} \circ d_{\cD} : C^{\infty}_c(M) \to C^{\infty}_c(M)$ is called the \emph{horizontal Laplacian} of the distribution $(M,\cD)$.
\end{definition}

\begin{remark}
The operator $\Delta_{\cD}$ can be described using the associated quadratic form (an analogue of the Dirichlet form): 
\[
(\Delta_{\cD}u,u)=\int_M \|d_{\cD}u(x)\|_{\cD^*_x}^{2}d\mu(x),\quad u \in C^{\infty}_c(M).
\]
\end{remark}

\begin{remark} Locally, the horizontal Laplacian admits a ``sum of squares'' description: Given a local presentation $(E_U,\rho_U)$, choose an orthonormal frame $(\omega_1,\ldots,\omega_d)$ of $E_U$. Then the vector fields $\rho_{U}(\omega_1), \ldots, \rho_{U}(\omega_d)$ generate $\cD\left|_{U}\right.$ and the restriction $\Delta_{\cD,U}$ of $\Delta_{\cD}$ to $U$ is given by
\begin{equation}\label{e:sum_of_sq}
\Delta_{\cD,U} = \sum_{i=1}^d \rho_{U}(\omega_i)^{\ast}\rho_{U}(\omega_i).
\end{equation} 
In particular, we see that $\Delta_{\cD}$ is a second order differential operator.
\end{remark}

\begin{remark}In the case $\cD$ is a foliation, $\Delta_{\cD}$ is a longitudinally elliptic operator. We refer the reader to \cite{umn-survey} for a survey of longitudinally elliptic operators on regular foliations and to \cite{Andr14,AS2} for the case of singular foliations.
\end{remark}

From now on, we restrict to the case where $M$ is a compact manifold. 

\begin{theorem} \label{t:ss}
The horizontal Laplacian $\Delta_{\cD}$, as an unbounded operator on the Hilbert space $L^2(M,\mu)$, with domain $C^{\infty}(M)$, is essentially self-adjoint.
\end{theorem}

Theorem~\ref{t:ss} is proved in \cite{YK1,AK1}, using a well-known result by Chernoff \cite{Chernoff} based on some facts from theory of linear symmetric first order hyperbolic systems, in particular, using in an essential way the fact of finite propagation speed of wave solutions of such equations. This follows from the compactness of $M$.

\section{Longitudinal hypoellipticity}
Under rather weak assumptions, one can associate with an arbitrary distribution $\cD$ a singular foliation $\cF=\cU(\cD)$, which includes $\cD$,  so that the horizontal Laplacian $\Delta_\cD$ is a longitudinal differential operator with respect to $\cF$. By definition, $\cD$ satisfies a kind of bracket generating condition with respect to $\cF$ and, therefore, induces a structure of sub-Riemannian manifold on each leaf of $\cF$. Therefore, one may expect that $\Delta_\cD$ is longitudinally hypoelliptic with respect $\cF$ that is justified by the results given in this section.

Recall that the vector space $\cX_c(M)$ carries two natural structures: the structure of $C^\infty(M)$-module given by the pointwise multiplication and the structure of Lie algebra given by the Lie bracket of vector fields. They satisfy some compatibility conditions. In short, one can say that $(C^\infty(M), \cX_c(M))$ is a Lie-Rinehart algebra in the sense of \cite{Rinehart}. 

Let $\cD$ be a $C^{\infty}(M)$-submodule of $\cX_c(M)$. The \emph{Lie-Rinehart subalgebra of $(C^\infty(M), \cX_c(M))$} associated to $\cD$ is the minimal submodule $\cU(\cD)$ of $\cX_c(M)$ which contains $\cD$ and is involutive, namely it satisfies $[X,Y] \in \cU(\cD)$ for every $X, Y \in \cU(\cD)$. Specifically, $\cU(\cD)$ is the $C^{\infty}(M)$-submodule of $\cX_{c}(M)$ generated by the elements of $\cD$ and their iterated Lie brackets $[X_1,\ldots,[X_{k-1},X_k]]$ with $X_i \in \cD$, $i = 1,\ldots,k$, for every $k \in \N$.

Observe that the associative algebra over $C^{\infty}(M)$ of differential operators on $M$ generated by $\cD$ coincides with the corresponding  associative algebra over $C^{\infty}(M)$ of differential operators on $M$ generated by $\cU(\cD)$, that is, in the case when $\cU(\cD)$ is a foliation, with the algebra of longitudinal differential operators for $\cU(\cD)$. 

As can be seen from the following example, even if $\cD$ is locally finitely generated, the module $\cU(\cD)$, in general, may not be locally finitely generated. 

\begin{example}
Let $f \in C^{\infty}(\R^2)$ be defined by $f(x,y) = e^{-\frac{1}{x}}$ if $x >0$ and $f(x,y)=0$ if $x\leq 0$. Consider the smooth distribution $\cD$ on $\R^2$, which is the $C^{\infty}_c(\R^2)$-module generated by the vector fields $X = \partial_x$ and $Y=f\partial_y$. Note that $\cD$ is not involutive (indeed, $[X,Y]=-x^{-2}X$ and the function $g(x,y)=x^{-2}$ is obviously not in $C^{\infty}(\R^2)$) and $\cU(\cD)$ coincides with the distribution described in Example \ref{ex:nonfoliation}. So $\cU(\cD)$ is not (locally) finitely generated and, therefore, not a singular foliation.
\end{example}

In the sequel, we will always consider the case when $\cD$ is a smooth distribution and $\cF=\cU(\cD)$ is locally finitely generated. Then $\cF$ is a singular foliation. 

In \cite{AS1,AS2}, the first author and Skandalis extended to singular foliations the basic results of elliptic theory for longitudinal differential operators developed for regular foliations by Connes in \cite{Connes79}. In particular, classes of longitudinal pseudodifferential operators and associated scale of Sobolev spaces have been constructed. In \cite{YK1}, the second author used the methods developed in \cite{AS1,AS2} as well as the methods of the study of hypoelliptic H\"ormander operators of sum of squares type to prove some basic properties of the horizontal Laplacian for an arbitrary (constant rank) smooth distribution. In \cite{AK1}, these results were extended to the case of arbitrary generalized smooth distribution.

\begin{theorem}[\cite{YK1,AK1}]\label{t:Hs-hypo}
Suppose that $M$ is a compact manifold, $\cD$ is a smooth distribution on $M$ such that $\cF=\cU(\cD)$ is a singular foliation and $\Delta_\cD$ is the horizontal Laplacian for $\cD$ associated with some choice of a Riemannian metric on $\cD$ and a positive smooth density $\mu$ on $M$. 

There exists $\epsilon>0$ such that, for any $s\in \mathbb R$, we have 
\[
\|u\|_{s+\epsilon}^2 \leq
C_s\left(\|\Delta_\cD u\|_s^2+\|u\|_s^2\right), \quad u\in C^\infty(M),
\]
where $C_s>0$ is some constant and $\|\cdot\|_s$ denotes a norm in the longitudinal Sobolev space $H^s(\mathcal F)$.
\end{theorem}

As a consequence, we immediately obtain the following result on longitudinal hypoellipticity. 

\begin{theorem}[\cite{YK1,AK1}]\label{t:hypo}
Under the assumptions of Theorem~\ref{t:Hs-hypo}, if $
u\in H^{-\infty}(\mathcal F):=\bigcup_{t\in \mathbb R}H^t(\mathcal F)$
such that $\Delta_\cD u\in H^s(\mathcal F)$ for some $s\in \mathbb R$, then $u \in H^{s+\varepsilon}(\mathcal F)$. 
\end{theorem}

These results, in particular, allow us to give another proof of essential self-adjointness of the operator $\Delta_\cD$ and also prove that, for any function  $\varphi$ from the Schwartz space $\mathcal S(\mathbb R)$, the operator $\varphi(\Delta_\cD)$ is leafwise smoothing with respect to the foliation $\mathcal F$, that is, it extends to a bounded operator from $H^s(\mathcal F)$ to $H^t(\mathcal F)$ for any $s,t\in \mathbb R$.

\section{The horizontal Laplacian as a multiplier}\label{s:multiplier}
As above, we will assume that $M$ is a compact manifold, $\cD$ is a smooth distribution on $M$ such that $\cF=\cU(\cD)$ is a singular foliation and $\Delta_\cD$ is the horizontal Laplacian for $\cD$ associated with some choice of a Riemannian metric on $\cD$ and a positive smooth density $\mu$ on $M$. 
\begin{assumption}\label{assumption}
From now on, we will assume that $\cF$ is a regular foliation. This means that $\mathcal F$ coincides with the subspace of smooth vector fields on $M$, tangent to leaves of some smooth foliation, which will be also denoted by $\mathcal F$. This assumption is justified\footnote{Using bisubmersions (\cf \cite{AS1}), as well as the methods developed in \cite{AS2} (where transversality plays an important role), it is not hard to lift Assumption \ref{assumption}. However, we do not know useful examples of distributions $\cD$ whose associated foliation $\cF$ is not regular.} in view of  the distributions arising in sub-Riemannian Geometry. In fact, in this case $\cF$ is the $C^{\infty}(M)$-module $\X(M)$ of all vector fields.
\end{assumption}

As mentioned above, $\Delta_\cD$ is a longitudinally hypoelliptic differential operator with respect to $\cF$. One can consider this operator as an unbounded operator on the Hilbert space $L^2(M,\mu)$, with domain $C^{\infty}(M)$. Then it is essentially self-adjoint. An important observation is that the operator $\Delta_D$ viewed as a formal differential expression gives rise to different unbounded operators, acting on essentially different functional spaces. For instance, we can associate with $\Delta_\cD$ a family $\{\Delta_L : L\in M/\mathcal F\}$ of self-adjoint differential operators on the leaves of the foliation $\mathcal F$ (or better on the holonomy coverings of the leaves). It turns out that, in spite of the fact that the operators act on quite different spaces, their spectral properties may be closely related, if we choose the operators $\Delta_L$ in an appropriate way.

Since $\Delta_\cD$ is a longitudinal differential operator for $\cF$, it can be restricted to each leaf $L$ of the foliation $\cF$. If the density $\mu$ is holonomy invariant with respect to $\cF$ and some smooth positive leafwise density $\alpha$, then one can take $\Delta_L$ to be the restriction of $\Delta_\cD$ to $L$. In this case, $\Delta_L$ can be described as the horizontal Laplacian on $L$. Indeed, one can define the restriction $\cD_L$ of the distribution $\cD$ to each leaf $L$ of $\mathcal F$. It is easy to see that the distribution $\cD_L$ is completely non-integrable (bracket-generating). We also have the restriction of the Riemannian structure on $\cD$ to $\cD_L$ and the fixed positive density $\alpha$ on $L$. Then the operator $\Delta_L$ is the horizontal Laplacian for $\cD_L$ associated with these data. In the general case, one should take into account the fact that the transverse part of $\mu$ is not constant along the leaves of $\cF$ (see the modular function $\delta$ and Definition \ref{defn:Rx} below). 

Here methods of operator algebras and noncommutative geometry are very useful. They have been developed for regular foliations by Connes \cite{Connes79,Connes80} (see \cite{survey,umn-survey} for more information) and for singular foliations by the first author and Skandalis \cite{AS1,AS2}. Their applications rely on the key observation that one can define an unbounded multiplier $P_\cD$ on the full $C^*$-algebra $C^*(\mathcal F)$ of the foliation $\cF$ such that both the operator $\Delta_\cD$ and the family $\{\Delta_L : L\in M/\mathcal F\}$ are the images of $P_\cD$ under suitable representations of $C^*(\mathcal F)$. These results were extended in \cite{Vassout} to the case of elliptic differential operators on Lie groupoids and in \cite{AS2} to the case of elliptic differential operators on singular foliations. 

Let us recall some necessary information on noncommutative geometry of regular foliations (for more information and details, see \cite{survey,umn-survey} and references therein).

Let $G$ be the holonomy groupoid of $\cF$. We will denote by $r:G\rightarrow M$ and $s:G\rightarrow M$ its range and source maps. 
Let $\alpha \in C^\infty(M,|T{\mathcal F}|)$ be an arbitrary smooth positive leafwise density on $M$. For any $x\in M$, define a smooth positive density $\nu ^{x}$ on $G^x: =r^{-1}(x)$ as the lift of the density $\alpha $ by the holonomy covering map $s:G^x\to M$. The family $\{\nu^x:x\in M\}$ is a smooth Haar system on $G$. 

The structure of involutive algebra on $C^{\infty}_c(G)$ is given by
\begin{align*}
k_1\ast k_2(\gamma)&=\int_{G^x} k_1(\gamma_1)
k_2(\gamma^{-1}_1\gamma)\,d\nu^x(\gamma_1),\quad \gamma\in G^x,\\
k^*(\gamma)&=\overline{k(\gamma^{-1})}, \quad \gamma\in G.
\end{align*}
We will denote by $C^{\ast}(\cF)$ (resp. $C^{\ast}_r(\cF)$) the full (resp. the reduced) $C^{\ast}$-algebra of the groupoid $G$. They are defined as suitable completions of $C^{\infty}_c(G)$.  

Let $\mu$ be a smooth positive density on $M$. Following \cite{renault,Fack-Skandalis}, we define a natural $\ast$-representation $R_\mu$ of the $C^*$-algebra $C^*(\mathcal F)$ in the Hilbert space $L^2(M,\mu)$. First, we observe that there exists a smooth non-vanishing function $\delta$ on $G$ such that, for any $f\in C_c(G)$, 
\[
\int_M\left(\int_{G^x}\delta(\gamma) f(\gamma^{-1})d\nu^x(\gamma)\right)d\mu(x)=\int_M\left(\int_{G^x}f(\gamma)d\nu^x(\gamma)\right)d\mu(x).
\]
In terminology of \cite{renault,Fack-Skandalis}, the function $\delta$ defines a homomorphism of the groupoid $G$ in the multiplicative group $\mathbb R_+$, and the measure $\mu$ on $M$ is a quasi-invariant measure of module $\delta$. Without loss of generality, we may assume that $\delta(x)=1$ for any $x\in M\subset G$.

\begin{definition}\label{defn:Rx}
For any $k\in C^\infty_c(G)$, the corresponding operator $R_\mu(k) : L^2(M,\mu)\to L^2(M,\mu)$ is defined for $u\in L^2(M,\mu)$ by
\[
R_\mu(k)u(x)=\int_{G^x} k(\gamma)\delta^{-\frac 12}(\gamma) u(s(\gamma))d\nu^x(\gamma), \quad x\in M.
\]
\end{definition}

One can give a local description of the operator $R_\mu(k)$. 
Let $\phi: \Omega \cong U\times T$ and $\phi': \Omega^\prime \cong U^\prime \times T$ be two compatible foliated charts on $M$ and $W(\phi,\phi')\subset G \stackrel{\cong}{\to} U\times U^\prime \times T$ the corresponding coordinate chart on $G$ \cite{Connes79} (see also \cite{survey,umn-survey}). Here $U,U^\prime\subset \R^p$ and $T\subset \R^q$ are open subsets, $p=\dim \cF$, $p+q=n$.
The restrictions $r : W(\phi,\phi') \to \Omega$ and $s : W(\phi,\phi') \to \Omega^\prime$ of the range and source maps to $W(\phi,\phi')\subset G$ are given by
\[
r(x,x^\prime,y)=(x,y),\quad s(x,x^\prime,y)=(x^\prime,y), \quad (x,x^\prime,y)\in U\times U^\prime \times T.
\]
In the charts $\phi$ and $\phi^\prime$, the density $\mu$ is written as $\mu=\mu(x,y)|dx||dy|$ and $\mu=\mu^\prime(x^\prime,y^\prime)|dx^\prime||dy^\prime|$, respectively, and the density $\alpha$ as $\alpha=\alpha(x,y)|dx|$ and $\alpha=\alpha^\prime(x^\prime,y^\prime)|dx^\prime|$, respectively. Then $\delta\in C^\infty(U\times U\times T)$ is given by (see \cite[Proposition VIII.12]{Connes79})
\[
\delta(x,x^\prime,y)=\frac{\mu(x,y)\alpha^\prime(x^\prime,y)}{\mu^\prime(x^\prime,y)\alpha(x,y)}, \quad (x,x^\prime,y)\in U\times U^\prime \times T.
\]
For any $k$ supported in $W$, $k\in C^\infty_c(W)\cong C^\infty_c(U\times U^\prime \times T)$, the operator $R_\mu(k) : C^\infty(\Omega^\prime)\to C^\infty(\Omega)$ has the form
\begin{multline*}
R_\mu(k)u(x,y)=\\ \int k(x,x^\prime,y)\left(\frac{\mu^\prime(x^\prime,y)}{\mu(x,y)}\right)^{1/2}(\alpha(x,y))^{1/2}(\alpha^\prime(x^\prime,y))^{1/2} u(x^\prime,y)dx^\prime.
\end{multline*}

Let $\cD$ be a smooth distribution on $M$ such that $\cF = \cU(\cD)$ is a regular foliation and $\Delta_{\cD}$ be the horizontal Laplacian of the distribution $(M,\cD)$ associated with some choice of a Riemannian metric on $\cD$ and a positive smooth density $\mu$ on $M$. Now we construct a pseudodifferential multiplier $P_{\cD}$ on $C^{\ast}(\cF)$, in the sense of \cite{AS2}, such that the horizontal Laplacian $\Delta_{\cD}$ is the image of $P_{\cD}$ by the representation $R_\mu$.

Let $M=\bigcup_{\alpha=1}^m U_\alpha$ be a finite open covering of $M$ such that, for any $\alpha=1,\ldots,m$, there exist a local representation $(E_{U_\alpha},\rho_{U_\alpha})$ and a local orthonormal frame $(\omega^{(\alpha)}_1,\ldots,\omega^{(\alpha)}_{d_\alpha})$ of $E_{U_\alpha}$. Then (see \eqref{e:sum_of_sq}) the restriction of $\Delta_\cD$ to $U_\alpha$ is written as
\[
\Delta_\cD\left|_{U_\alpha}\right.=\sum_{j=1}^{d_\alpha}(X^{(\alpha)}_j)^* X^{(\alpha)}_j,
\]
where $X^{(\alpha)}_j=\rho_{U_\alpha}(\omega^{(\alpha)}_j)\in \cD\left|_{U_\alpha}\right.$, $j=1,\ldots, d_\alpha$.

Let $\phi_\alpha\in C^\infty(M)$ be a partition of unity subordinate
to the covering, ${\rm supp}\,\phi_\alpha\subset U_\alpha$, and
$\psi_\alpha\in C^\infty(M)$ such that ${\rm
supp}\,\psi_\alpha\subset U_\alpha$,
$\phi_\alpha\psi_\alpha=\phi_\alpha$. Then we have
\[
\Delta_\cD=\sum_{\alpha=1}^m \phi_\alpha (\Delta_\cD\left|_{U_\alpha}\right.)\psi_\alpha=\sum_{\alpha=1}^m\sum_{j=1}^{d_\alpha} \phi_\alpha (X^{(\alpha)}_j)^* X^{(\alpha)}_j\psi_\alpha.
\]
Now, from \cite{AS2} (or \cite{YK1}) we know that each $X\in \cF$ is the image of some multiplier $L_X\in \Psi^1(\cF)$ and, since $R_\mu$ is a $\ast$-presentation, each $X^*\in \cF$ is the image of some multiplier $L_{X^*}\in \Psi^1(\cF)$. Let $s^*X$ (resp. $r^*X$) be the unique vector field on $G$ such that $ds_\gamma(s^* X(\gamma)) =X(s(\gamma))$ and $dr_\gamma(s^* X(\gamma))=0$ (resp. $ds_\gamma(r^* X(\gamma)) =0$ and $dr_\gamma(s^* X(\gamma))=X(r(\gamma))$) for any $\gamma\in G$. It is easy to see that, for $k_1,k_2\in C^{\infty}_c(G)$, we have
\[
s^*X(k_1\ast k_2)=k_1\ast s^*X(k_2), \quad r^*X(k_1\ast k_2)=r^*X(k_1)\ast k_2.
\]
So $s^*X$ is $C^{\infty}_c(G))$-linear with respect to the left multiplication by the elements of $C^{\infty}_c(G)$ and $r^*X$ is $C^{\infty}_c(G))$-linear with respect to the right multiplication by the elements of $C^{\infty}_c(G)$. Both $L_X$  and $L_{X^*}$ are first order differential operators on $G$ of the form 
\[
L_X=r^*X+r^*l_X, \quad L_{X^*}=-r^*X+r^*\tilde l_X,
\]
with some $l_X, \tilde l_X\in C^\infty(M)$ such that
\[
XR_\mu(k)=R_\mu(L_X k), \quad X^*R_\mu(k)=R_\mu(L_{X^*}k),\quad  k\in C^\infty_c(G).
\]
Therefore, if we define $P_\cD$ to be the second order differential operator on $G$ given by
\begin{equation}\label{e:PDelta}
P_\cD=\sum_{\alpha=1}^m\sum_{j=1}^{d_\alpha} r^*\phi_\alpha L_{(X^{(\alpha)}_j)^*} L_{X^{(\alpha)}_j} r^*\psi_\alpha,
\end{equation}
then the operator $\Delta_\cD$ is the image of $P_\cD$ under the representation $R_\mu$, that is, 
\[
\Delta_\cD R_\mu(k)=R_\mu(P_\cD k), \quad  k\in C^\infty_c(G).
\]

Note that the above also works for noncompact manifolds, because in this case all the sums are infinite, but locally finite, and for singular foliations \cite{AK1}.

Recall that a Hilbert module over a $C^*$-algebra $A$ is a right $A$-module $E$ endowed with a positive definite sesquilinear map $\langle \cdot, \cdot\rangle : E\times E\to A$ such that the $\|x\|_E=\|\langle x,x\rangle \|_A^{1/2}$ equips $E$ with a structure of Banach space. An unbounded $A$-linear operator $T$ on a Hilbert module $E$ is called regular, if it is densely defined, its adjoint is densely defined, and its graph admits an orthogonal complement, which means that $A\oplus A = \Gamma \oplus \Gamma^\bot$, where $\Gamma = \{(x, Tx) : x \in \operatorname{Dom} T\}$ is the graph of $T$ and $\Gamma^\bot = \{(T^*y,-y) :  y \in \operatorname{Dom} T^*\}$ is its orthogonal complement with respect to an obvious $A$-valued inner product on $A\oplus A$.

The notion of unbounded regular operator on a Hilbert module over a  $C^*$-algebra was introduced by Baaj in his thesis \cite{Baaj80} (see also  \cite{Baaj-Julg}). Regular operators have many nice properties, similar to the properties of closed densely defined operators in a Hilbert space. In particular, for self-adjoint regular operators, there is a continuous functional calculus (see, for instance, \cite{Lance}). 

Let us consider the $C^*$-algebra $C^*(\mathcal F)$ as a Hilbert module over itself, the right module structure is given by the right multiplication by elements of the algebra and the inner product by $\langle a,b\rangle=a^*b$.  An unbounded regular operator on this module is also called an unbounded multiplier on $C^*(\mathcal F)$.
  
We will consider the operator $P_\cD$ as an unbounded, densely defined operator on the Hilbert module $C^*(\mathcal F)$ with domain $\mathcal A=C^\infty_c(G)$. Using the fact that $R_\mu$ is a $\ast$-representation of $C^*(\mathcal F)$, injective on $C^\infty_c(G)$, one can show that $P_\cD$ is formally self-adjoint, that is, $\langle P_\cD k_1,k_2\rangle=\langle k_1, P_\cD k_2\rangle $ for any $k_1,k_2\in C^\infty_c(G)$.
Since $P_\cD$ is densely defined, its formal self-adjointness immediately implies the existence of the closure $\overline{P_\cD}$. 
 
\begin{theorem}\label{t:multi2} 
The operator $\overline{P_\cD}$ is an unbounded multiplier of $C^*(\mathcal F)$.
\end{theorem}

The proof of Theorem~\ref{t:multi2} will be given in Section~\ref{s:param}. In the case when $\cD$ is a smooth distribution of constant rank, it was given in \cite{YK2}.

\section{Leafwise representations}

For any $x\in M$, there is a natural representation of
$C^{\infty}_c(G)$ in the Hilbert space $L^2(G^x,\nu^x)$ given, for
$k\in C^{\infty}_c(G)$ and $\zeta \in L^2(G^x,\nu^x)$, by
\begin{equation}\label{e:Rx}
R_x(k)\zeta(\gamma)=\int_{G^x}k(\gamma^{-1}\gamma_1)
\zeta(\gamma_1) d\nu^x(\gamma_1),\quad r(\gamma)=x.
\end{equation}

Let us compute the image of $P_\cD$ under the representation $R_x$, $x\in M$. It is a differential operator $\Delta_x$ on $G^x$ such that, for any $k\in C^\infty_c(G)$,  
\[
R_x(P_\cD k)=\Delta_x R_x(k).
\]
For any $k\in C^\infty_c(G)$, the family $\{R_x(k), x\in M\}$ defines an operator $R(k)$ on $C^\infty_c(G)$, which is $C^\infty_c(G)$-linear with respect to the left multiplication. Therefore, we just need to switch from the operator $P_\cD$ on $C^\infty_c(G)$, which is $C^\infty_c(G)$-linear with respect to the right multiplication, to an operator on $C^\infty_c(G)$ defined by a family $\{\Delta_x, x\in M\}$ of differential operators on $G^x$, which is $C^\infty_c(G)$-linear with respect to the left multiplication. 
 
For any $k\in C^\infty(G)$, define a function $\tilde k \in C^\infty(G)$ by $\tilde k(\gamma)=k(\gamma^{-1}), \gamma\in G$. It is easy to check that, for any vector field $X\in C^\infty(M,T\mathcal F)$ and any function $a\in C^\infty(M)$, we have the identities:
\[
\widetilde{(r^*X)k}=(s^*X)\tilde k, \quad \widetilde{(r^*a)k}=(s^*a)\tilde k, \quad k\in C^\infty(G).
\]
Using these identities, from \eqref{e:PDelta} and \eqref{e:Rx}, we easily get 
\[
\Delta_x=\sum_{\alpha=1}^d\sum_{j=1}^p  s^*\phi_\alpha\tilde R_{X^{(\alpha)}_j} R_{X^{(\alpha)}_j}s^*\psi_\alpha,
\]
where
\[
R_X=s^*X+s^*l_X, \quad \tilde R_X=-s^*X+s^*\tilde l_X.
\] 

Consider the longitudinally elliptic operator $\Delta_M$ on $M$ given by
\[
\Delta_M=\sum_{\alpha=1}^M\sum_{j=1}^d  \phi_\alpha (-{X^{(\alpha)}_j}+\tilde l_{X^{(\alpha)}_j}) (X^{(\alpha)}_j+l_{X^{(\alpha)}_j}) \psi_\alpha.
\]
If we restrict this operator to the leaf $L_x$ through $x\in M$ and then lift it to the holonomy covering $G^x$ by use of the map $s:G^x\to L_x$, then we get the operator $\Delta_x$. Remark that the operator $\Delta_M$, in general, does not equal $\Delta_\cD$, but they have the same principal symbol and  coincide when the density $\mu$ is holonomy invariant with respect to $\alpha$ (or equivalently $\delta\equiv 1$). 

As a straightforward consequence of Theorem \ref{t:multi2}, we obtain in a standard way (cf., for instance, \cite{Kord95,Vassout}) the following statement.  

\begin{theorem}\label{t:spectrum}
Denote by $\sigma_{\mathcal  F}(\Delta_\cD)$ the leafwise spectrum of $\Delta_\cD$:
\[
\sigma_{\mathcal  F}(\Delta_\cD)=\overline{\bigcup \{\sigma(\Delta_x): x \in M\}},
\]
where $\sigma(\Delta_x)$ is the spectrum of $\Delta_x$ in $L^{2}(G^x,\nu^x)$, and by $\sigma(\Delta_\cD)$ the spectrum of $\Delta_\cD$ in $L^{2}(M,\mu)$. Then:
\begin{itemize}
\item $\sigma_{\mathcal  F}(\Delta_\cD)\subset \sigma(\Delta_\cD)$;
\item If the holonomy groupoid is amenable (that is, $C^*(\mathcal F)\cong C^*_r(\mathcal F)$), then $\sigma(\Delta_\cD)=\sigma_{\mathcal  F}(\Delta_\cD)$.
\end{itemize}
\end{theorem} 

This result was proved in \cite{Kord95} when $\cD$ is a regular foliation and in \cite{AS2} (see also \cite{Andr14}) when $\cD$ is a singular foliation.

 \section{Construction of a parametrix}\label{s:param}

The proof of Theorem~\ref{t:multi2} is based on the following result \cite{Vassout,AS2,YK2}. 

\begin{theorem}\label{t:param}
Let $E$ be a Hilbert module over a $C^*$-algebra $A$ and $P$ be a an unbounded, formally self-adjoint operator on $E$ with dense domain $\mathcal A$. Suppose that $Q$, $R$ and $S$ are elements of $A$, considered as morphisms of the Hilbert module $E$, such that the following identities hold (on $\mathcal A$):  
\[
I-QP = R, \quad I-PQ = S.
\]
Moreover, suppose that the operators $PR$ and $PS^*$ extend to compact morphisms of $E$ and, therefore, belong to the algebra $A$. Then the operator $\overline{P}$ gives rise to an unbounded regular self-adjoint operator on $E$.
\end{theorem}

The rest of this section is devoted to a construction of a parametrix for the operator $P_\cD$, satisfying the conditions of Theorem~\ref{t:param}. 

\begin{theorem}\label{t:param1}
Let $(M,\cD)$ be a smooth distribution such that $\cF = \cU(\cD)$ is a regular foliation. Let $\Delta_{\cD}$ be the horizontal Laplacian of $(M,\cD)$ associated with some choice of a Riemannian metric on $\cD$ and a positive smooth density $\mu$ on $M$. There are $Q, R, S\in C^*(\mathcal F)$ such that   
\[
k-P_\cD(k)\ast Q  = k\ast R, \quad k-P_\cD (k\ast Q) = k\ast S, \quad k\in C^\infty_c(G).
\]
Moreover, the operators $P_\cD R$ and $P_\cD S^*$ extend to compact morphisms of the Hilbert module $C^*(\mathcal F)$ and, therefore, belong to the algebra $C^*(\mathcal F)$.
\end{theorem}

\begin{proof}
First, we consider the local setting. Let $\Omega\cong U\times T$, $U\subset \mathbb R^p$,  $T\subset \mathbb R^q$, be an arbitrary foliated coordinate neighborhood such that there exist a local representation $(E_{\Omega},\rho_{\Omega})$ defined over $\Omega$ and a local orthonormal frame $(\omega_1,\ldots,\omega_{d})$ of $E_{\Omega}$. Then (see \eqref{e:sum_of_sq}) the restriction of $\Delta_\cD$ to $\Omega$ is written as
\begin{equation}\label{e:DeltaOmega}
\Delta_\cD\left|_{\Omega}\right.=\sum_{j=1}^{d}(X_j)^* X_j,
\end{equation}
where $X_j=\rho_{\Omega}(\omega_j)\in \cD\left|_{\Omega}\right.$, $j=1,\ldots, d$.

Since the vector fields $X_j$ are tangent to the foliation $\mathcal F$, they are tangent to the plaques $U\times \{y\}$ of $\Omega$, and, therefore, $X_j(x,y)\in \mathbb R^p\cong \mathbb R^p\oplus \{0\}\subset \mathbb R^p\oplus \mathbb R^q$ for any $(x,y)\in U\times T$. In particular, any $X_j$ is given by a family $\{X_{j,y}, y\in T\}$ of vector fields on $U$. It is easy to see that, for any $y\in T$, the family $\{X_{j,y}, j=1,\ldots,d\}$ of vector fields  on $U$ is bracket-generating.

For any function $a\in C^\infty(U\times T)$, we will denote by $a_y\in C^\infty(U\times \{y\})\cong C^\infty(U)$ its restriction to $U\times \{y\}, y\in T$. It is easy to check that the restriction of the operator $\Delta_M$ to $\Omega$ is given by a smooth family $\{\Delta_y, y\in T\}$ of second order differential operators on $U$ of the form
\begin{equation}\label{e:Delta-local}
\Delta_y=-\sum_{j=1}^d X^2_{j,y}+\sum_{j=1}^da_{j,y}X_{j,y}+b_y, \quad y\in T, 
\end{equation}
where $a_j,b\in C^\infty(U\times T)$.

A crucial role in the proof of Theorem~\ref{t:param1} is played by the following fact. 

\begin{theorem}\label{t:param-local}
Let $\{\Delta_y, y\in T\}$ be a smooth family of second order differential operator on $U$ of the form \eqref{e:Delta-local}, where, for any $y\in T$, the family $\{X_{j,y}, j=1,\ldots,d\}$ of vector fields on $U$ is bracket-generating.

For any $(x_0,y_0)\in \Omega$, there exist neighborhoods $U_0\subset U$ of $x_0$ and $T_0\subset T$ of $y_0$ such that, for any $\phi\in C^\infty_c(U_0\times T_0)$, there exists a family $\{Q_y, y\in T_0\}$ of compact operators in $L^2(U_0)$, continuous in the uniform operator topology, with the Schwartz kernel compactly supported in $U_0\times U_0\times T_0$, such that
\[
Q_y\Delta_y=\phi_y I-R_y,\quad \Delta_y Q_y=\phi_y I-S_y,\quad y\in T_0,
\]
where the operators $R_y$, $S_y$, $\Delta_y R_y$ and $\Delta_y S^*_y$ on $C^\infty(U_0)$ extend to compact operators in $L^2(U_0)$, depending continuously on $y\in T_0$ in the uniform operator topology.
\end{theorem}

The proof of Theorem~\ref{t:param-local} is given in \cite{YK2}. It uses a classical parametrix construction for hypoelliptic H\"ormander sum of squares type operators given in the paper of Rothschild and Stein \cite{Rothschild-Stein}. The main difficulty is to prove that, given a smooth family of hypoelliptic H\"ormander sum of squares type operators, one can   construct a smooth family of their parametrices. For this, we follow the constructions of \cite{Rothschild-Stein}, checking in the process the smooth dependence of all constructed objects on the family parameter. In \cite{YK2}, we consider the case when the distribution $\cD$ is of constant rank and assume that the vector fields $X_j, j=1,\ldots,d,$ in \eqref{e:DeltaOmega} are linearly independent. It is easy to see that this proof can be easily extended to the current setting, first of all, because the initial step in the proof, the Rothschild-Stein lifting theorem and its proof given in \cite{Hoermander-Melin}, hold in this generality. 

Let us turn to the global setting. For convenience of notation, we switch from the operator $P_\cD$, which is $C^\infty_c(G)$-linear with respect to the right multiplication, to the operator $\{\Delta_x, x\in M\}$, which is $C^\infty_c(G)$-linear with respect to the left multiplication. 

Recall the notion of $G$-operator introduced in \cite{Connes79}. For any $\gamma\in G$, $\gamma : x\to y$, define the left translation operator $L(\gamma) : C^\infty(G^x)\to C^\infty(G^x)$ by
\[
L(\gamma)f(\gamma^\prime)=f(\gamma^{-1}\gamma^\prime), \quad \gamma^\prime \in G^y.
\]
A $G$-operator is a family $\{P_x, x\in M\}$, where $P_x$ is a linear continuous map in $C^\infty_c(G^x)$, which is left-invariant: $L(\gamma)\circ P_x=P_y\circ L(\gamma)$ for any $\gamma : x\to y$. It is easy to see that $\{P_x, x\in M\}$ is a $G$-operator if and only if the corresponding operator in $C^\infty_c(G)$ is $C^\infty_c(G)$-linear with respect to the left multiplication. In particular, the family $\{\Delta_x, x\in M\}$ is a $G$-operator, and, for any $k\in C^\infty_c(G)$,  the family $\{R_x(k), x\in M\}$ given by \eqref{e:Rx} is a $G$-operator

The following construction given in \cite{Connes79} (see, in particular,  \cite[Proposition VIII.7b)]{Connes79}) allows one to construct $G$-operators, starting from a continuous family of integral operators, defined in a foliated coordinate neighborhood. Let $\Omega\cong U\times T$ be a foliated coordinate neighborhood and $\{P_y : y\in T\}$ be a continuous family of operators in $C^\infty(U)\cong C^\infty(U\times \{y\})$. Assume that the Schwartz kernel $k_y\in C^{-\infty}(U\times U)$ is compactly supported in $U\times U\times T$. (In this case, we will say that the family  $\{P_y\}$ is compactly supported in $U\times T$.) A natural embedding of $W=U\times U\times T$ into $G$ allows one to consider $k$ as a distribution on $G$ and, therefore, define the corresponding $G$-operator $\{P^\prime_x=R_x(k) : x\in M\}$ by the formula \eqref{e:Rx}. This operator can be also described as follows. For an arbitrary $x\in M$, we define an equivalence relation on $G^x\cap s^{-1}(\Omega)$, setting $\gamma_1\sim \gamma_2$, if $\gamma_1^{-1}\gamma_2\in W$. Each equivalence class $\ell$ is open (since $W$ is open) and connected (since $U$ is connected). Therefore, $\ell$ is a connected component and $G^x\cap s^{-1}(\Omega)=\cup \ell $ is the representation of $G^x\cap s^{-1}(\Omega)$ as the union of connected components. The restriction of $s$ to $\ell$ is a homeomorphism of $\ell$ on some plaque $s(\ell)=U\times\{y\}$ of $\Omega$. The kernel  $K_x(\gamma,\gamma_1)$ of the operator $P^\prime_x : C^\infty(G^x)\to C^\infty(G^x)$ may be different from zero only if $\gamma$ and $\gamma_1$ are in the same component of $G^x\cap s^{-1}(\Omega)$, and the restriction of the operator to the component $\ell$ corresponds under the map $s$ to an operator $P_y$, acting on $s(\ell)=U\times\{y\}$. 

By \cite[Proposition VIII.7b)]{Connes79}, if the operator $P_y$ is compact in $L^2(U)$ for any $y\in T$  and the map $y\to P_y$ is continuous in the uniform operator norm in $L^2(U)$, then the corresponding $G$-operator $P^\prime$ belongs to $C^*(\mathcal F)$. 

The proof of Theorem~\ref{t:param1} is completed by means of the standard gluing construction of local parametrices (cf. \cite[Proposition IX.2)]{Connes79}). Let $M=\cup_i\Omega_i$ be a finite covering of $M$ by foliated coordinate neighborhoods, for each of which the statement of Theorem~\ref{t:param-local} holds, $\Omega_i\cong U_i\times T_i$, $\phi_i$ be a partition of unity, subordinate to this covering, $\psi_i\in C^\infty_c(\Omega_i)$ be functions such that $\psi_i=1$ on the support of $\phi_i$. Let $\{\Delta_{i,y} : y\in T_i\}$ be a smooth family of differential operators on  $U_i$ defined by \eqref{e:Delta-local} in a foliated coordinate neighborhood $\Omega_i\cong U_i\times T_i$. Observe that the $G$-operator $(\psi_i\Delta_i)^\prime$ obtained by use of the above construction from the family $\{\psi_{i,y}\Delta_{i,y}: y\in T_i\}$, coincides with the differential $G$-operator $\{s^*\psi\Delta_x: x\in M\}$. By Theorem~\ref{t:param-local}, for any $i$, there is a continuous family $\{Q_{i,y} : y\in T_i\}$ of compact operators in $L^2(U_i)$ with compact support with $U_i\times T_i$ such that 
\[
Q_{i,y}\Delta_{i,y}=\phi_{i,y} I-R_{i,y},\quad \Delta_{i,y} Q_{i,y}=\phi_{i,y} I-S_{i,y},\quad y\in T_i,
\]
where the operators $R_{i,y}$, $S_{i,y}$, $\Delta_{i,y} R_{i,y}$ and $\Delta_{i,y} S^*_{i,y}$ on $C^\infty(U_i)$ extend to compact operators in $L^2(U_i)$, continuously depending on $y\in T$ in the uniform operator topology. It is easy to check that the $G$-operator $Q=\sum_i Q^\prime_{l,i}(\psi_i\circ s)$ is a desired one. 
\end{proof}

\end{document}